\documentclass[a4paper,11pt]{article}
\usepackage{amsmath}
\usepackage{amsfonts}
\usepackage{titlesec}
\usepackage{amssymb}
\usepackage{amsthm}

\usepackage{geometry}
\usepackage{color}

\newcommand{\uzero}{u_0}

\newcommand{\R}{\mathbb{R}}
\newcommand{\C}{\mathbb{C}}

\usepackage[utf8]{inputenc}

\usepackage{mathtools}

\usepackage{array}
\makeatletter
\newcommand{\thickhline}{%
	\noalign {\ifnum 0=`}\fi \hrule height 1pt
	\futurelet \reserved@a \@xhline
}
\newcolumntype{"}{@{\hskip\tabcolsep\vrule width 1pt\hskip\tabcolsep}}
\makeatother

\newcommand{\ci}{\mathrm{i}} 
\newcommand{\deltat}[1]{\tau_{#1}}
\newcommand{\hatu}[1]{u_{\deltat{}}^{#1}}

\usepackage{framed}
\usepackage{graphicx, caption, subcaption}
\usepackage{hyperref}
\usepackage{dsfont}
\usepackage{float}

\usepackage[titlenumbered,ruled]{algorithm2e}

\newtheorem{theorem}{Theorem}[section]

\newtheorem{lemma}[theorem]{Lemma}
\newtheorem{remark}[theorem]{Remark}

\renewcommand{\thefootnote}{\fnsymbol{footnote}}
\renewcommand{\thefootnote}{\arabic{footnote}}

\begin{document}

\begin{center}
{\LARGE A note on optimal $H^1$-error estimates for Crank--Nicolson approximations to the nonlinear Schr\"odinger equation\renewcommand{\thefootnote}{\fnsymbol{footnote}}\setcounter{footnote}{0}
 \hspace{-3pt}\footnote{The authors acknowledge the support by the Swedish Research Council (grant 2016-03339) and the G\"oran Gustafsson foundation.}}\\[2em]
\end{center}

\begin{center}
{\large Patrick Henning and Johan W\"arneg{\aa}rd \footnote[1]{Department of Mathematics, KTH Royal Institute of Technology, SE-100 44 Stockholm, Sweden.}}\\[2em]
\end{center}

\begin{center}
{\large{\today}}
\end{center}

\begin{center}
\end{center}

\begin{abstract}
In this paper we consider a mass- and energy--conserving Crank-Nicolson time discretization for a general class of nonlinear Schr\"odinger equations. This scheme, which enjoys popularity in the physics community due to its conservation properties, was already subject to several analytical and numerical studies. However, a proof of optimal $L^{\infty}(H^1)$-error estimates is still open, both in the semi-discrete Hilbert space setting, as well as in fully-discrete finite element settings. This paper aims at closing this gap in the literature.  We also suggest a fixed point iteration to solve the arising nonlinear system of equations that makes the method easy to implement and efficient. This is illustrated by numerical experiments.
\end{abstract}
	
\paragraph*{AMS subject classifications}
35Q55, 65M60, 65M15, 81Q05

\section{Introduction}

In this paper we consider nonlinear Schr\"odinger equations (NLS) seeking a complex function $u(t,x)$ such that
$$
\ci \partial_t u = - \triangle u + V u + \gamma( |u|^2 ) u
$$
in a bounded domain $\mathcal{D} \subset \mathbb{R}^d$, with a homogenous Dirichlet boundary condition on $\partial\mathcal{D}$ and a given initial value. Here, $V(x)$ is a known real-valued potential and $\gamma : [0,\infty) \rightarrow \mathbb{R}$ is a smooth (and possibly nonlinear) function that depends on the unknown density $|u|^2$. Of particular interest are cubic nonlinearities of the form $\gamma(|u|^2)u=\kappa |u|^2u$, for some $\kappa \in \mathbb{R}$. In this case, the equation is called Gross--Pitaevskii equation. It has applications in optics \cite{Optics,GravityWaves}, fluid dynamics \cite{FluidReview,DeepWater} and, most importantly, in quantum physics, where it models for example the dynamics of Bose-Einstein condensates in a magnetic trapping potential \cite{Gro61,LSY01,Pit61}. Another relevant class is saturated nonlinearities, such as $\gamma(|u|^2)=\kappa |u|^2 (1 + \alpha |u|^2)^{-1}$ for some $\alpha\ge 0$, which appear in the context of nonlinear optical wave propagation in layered metallic structures \cite{MNF89,RTZ00} or the propagation of light beams in plasmas \cite{Max76}.
In order to discretize nonlinear Schr\"odinger equations in time, splitting methods and exponential integrators  typically yield highly efficient solution schemes that can be easily combined with a spectral discretization in space (cf. \cite{BBD02,CaG15,CoG12,Fao12,HLW06,Lub08,ThA12,Tha12b} and the references therein).
If the exact solution to the NLS admits high regularity, such discretization schemes typically show a remarkably good performance. However, if the regularity of the solution is strongly reduced, either by rough potentials $V$ (e.g. disorder potentials or optical lattices) or by rough initial values $u(0)$ (e.g. when effects close to phase transitions are studied), then the performance of these methods can drop dramatically. Here we refer exemplarily to the recent numerical experiments reported in \cite{OstermannSchratz-1,OstermannSchratz-2,HW18}. To overcome this issue, Ostermann and Schratz proposed new low-regularity time-integrators \cite{OstermannSchratz-1,OstermannSchratz-2} which improve the convergence in low regularity regimes significantly. However, the approach still relies on a Fourier discretization in space, which is not an optimal choice due to the loss of spectral convergence for non-smooth solutions. Practically, the usage of a (low order) finite element space discretization is often desirable in order to account for spatial low regularity. In the following we will only discuss approaches that can be easily combined with finite elements in space, meaning that we put ourselves into the situation that we assume that the solution to the NLS does not admit much smoothness.

Nonlinear Schr\"odinger equations come with important physical invariants, where the mass and the energy are considered as two of the most crucial ones. When solving a NLS numerically it is therefore of great importance to also reproduce this conservation on the discrete level. This aspect was emphasized by various numerical studies \cite{HW18,Sanz-Serna}, where it was also found that the complexity of the physical setup (or low-regularity) can stress this issue even further.

For the subclass of power law nonlinearities of the form $\gamma(|u|^2)= \sum_{k=1}^K \kappa_k |u|^{2\sigma_k}$ for $\sigma_k \ge 0$ and $\alpha_k \in \R$, a mass and energy conserving relaxation scheme was proposed and analyzed by Besse \cite{Besse,BDD18}. Thanks to its properties, the scheme shows a very good performance in realistic physical setups \cite{HW18}. Despite the large variety of different numerical approaches for solving the time-dependent NLS (cf. \cite{Akrivis1991,ANTOINE20132621,BaoNum,HeM17,KaM98,KaM99,Lub08,Sanz-SernaNLCN,ThA12,Tha12b,CrankH1,Wang,zouraris_2001} and the references therein) the literature knows however of only one time discretization that conserves both mass and energy simultaneously for arbitrary (smooth) nonlinearities. This discretization, which was first mathematically studied by Sanz-Serna \cite{Sanz-SernaNLCN} and which is long-known in the physics community, is a Crank--Nicolson-type (CN) approach where the nonlinearity is approximated by a suitable difference quotient involving the primitive integral of $\gamma$. This is also the time discretization that we shall consider in this paper. Here we note it was analytically and numerically demonstrated that this method is applicable and reliable in low-regularity regimes \cite{HW18,NonlinearCN}. 

A combination of the method with a finite difference space discretization was proposed and analyzed by Bao and Cai  \cite{BaC12,BaC13}. Combining the Crank--Nicolson time discretization with a $P1$ finite element discretization in space, the first a priori error estimates 
for the arising method
were obtained by Sanz-Serna 1984 \cite{Sanz-SernaNLCN} for cubic nonlinearities. He considers the case $d=1$ and derives optimal $L^{\infty}(L^2)$-error estimates under the coupling constraint $\deltat{}\lesssim h$, where $\deltat{}$ denotes the time step size of the Crank-Nicolson method and $h$ the mesh size of the finite element  discretization in space. In 1991, Akrivis et al. \cite{Akrivis1991} improved this result by showing optimal convergence rates in $L^{\infty}(L^2)$ in dimension $d=1,2,3$ and under the relaxed coupling constraint $\deltat{} \lesssim h^{d/4}$. Finally, in 2017 \cite{NonlinearCN}, the $L^{\infty}(L^2)$-error estimates could be improved yet another time by showing that the coupling constraint can be fully removed. Furthermore, general nonlinearities could be considered, the influence of potentials could be taken into account  and even convergence under weak regularity assumptions could be proved (with reduced convergence rates). However, so far, optimal error estimates in $L^{\infty}(H^1)$ for this particular CN-discretization are still open in the literature. 

One reason for this absence of $H^1$-results could be related to the techniques used for the error analysis in previous works (cf. \cite{Akrivis1991,HeM17,KaM98,KaM99,Sanz-SernaNLCN,zouraris_2001}) which is based on the following steps: {\it 1.} Appropriate truncation of the nonlinearity to obtain a problem with bounded growth. {\it 2.} Analyzing the scheme with truncation in the FE space and deriving corresponding $L^{\infty}(L^2)$- and/or $L^{\infty}(H^1)$-error estimates. {\it 3.} Using inverse estimates in the finite element space to show that the truncated approximations are uniformly bounded in $L^{\infty}(L^{\infty})$ by a term of the form $C (1  +  h^{-s} ( \deltat{}^2 + h^{p}))$, with appropriate powers $p>s>0$ that depend on the considered space discretization, regularity and space dimension. {\it 4.} Concluding that if $\deltat{}$ and $h$ are coupled in an appropriate way, then the truncated approximations are all uniformly bounded by a constant $C$ and hence coincide with a solution to the scheme without truncation.

This strategy does not only have the disadvantage that it produces unnecessary coupling conditions, but also that it becomes impractically technical when considering $L^{\infty}(H^1)$-error estimates for the Crank--Nicolson FEM. This is because it requires a suitable truncation of the primitive integral of $\gamma$ that is on the one hand consistent with the energy conservation and on the other hand allows for uniform bounds of the approximations in $L^{\infty}(W^{1,\infty})$. However, thanks to the new techniques developed in \cite{Wang} and the CN error analysis suggested in \cite{NonlinearCN} in the context of $L^{\infty}(L^2)$-error estimates, the truncation step is no longer necessary and the desired $L^{\infty}(L^{\infty})$-bounds can be derived with elliptic regularity theory. With this, it is now possible to obtain $L^{\infty}(H^1)$ estimates in a direct way, not only in the finite element setting, but also in the semi-discrete Hilbert space setting.

In this paper we will therefore build upon the results from \cite{NonlinearCN,Wang} to fill the gap in the literature and prove optimal $L^{\infty}(H^1)$-error estimates for the energy-conservative Crank--Nicolson approach without coupling constraints and for a general class of nonlinearities. The paper is structured as follows. In Section \ref{section-notation-assumptions} we present the notation and the analytical assumptions on the problem. In Section \ref{section-time-discrete-CN} we present the time--discrete Crank--Nicolson method, we recall its well-posedness and optimal error estimates in $L^{\infty}(L^2)$. Furthermore, we present and prove the new error estimate in $L^{\infty}(H^1)$. 
The paper continues with the fully-discrete setting presented in Section \ref{section-fully-discrete-CN}, where the time discretization is combined with a finite element discretization in space. We recall what is known about this discretization and finally prove corresponding $L^{\infty}(H^1)$-error estimates, which is the main result of this paper. The paper concludes with a note on how to efficiently implement the method and two numerical experiments to confirm the convergence rates and to illustrate a setting in which it makes computational sense to use the CN-FEM instead of for example a spectral method.

\section{Notation and Assumptions}
\label{section-notation-assumptions}

We start with introducing the analytical setting of this work. Throughout the paper we assume that ${\mathcal{D}} \subset \R^d$ (for $d=2,3$) is a convex bounded domain with polyhedral boundary.  On $\mathcal{D}$, the Sobolev space of complex-valued, weakly differentiable functions with a zero trace on $\partial \mathcal{D}$ and $L^2$-integratable partial derivatives is as usual denoted by $H^1_0(\mathcal{D}):=H^1_0(\mathcal{D},\C)$. The potential $V \in L^{\infty}(\mathcal{D};\R)$ is assumed to be real and nonnegative. Indirectly, we also assume that $V$ is sufficiently smooth so that it is compatible with the regularity assumptions for $u$ listed below (see \cite{NonlinearCN} for a discussion on this aspect). The (possibly nonlinear) function 
$$\gamma : [0,\infty) \rightarrow [0,\infty)$$
is assumed to be $C^2[0,\infty)$, fulfills $\gamma(0)=0$ and its growth can be characterized with
\begin{align*} 
|\gamma(|v|^2)v - \gamma(|w|^2)w| \le L(K) |v - w| \qquad \mbox{for all } v,w \in \C \mbox{ with } |v|,|w|\le K
\end{align*}
where $L$ is a function with the following growth properties
\begin{align*} 
0 \le L(s) \le C s^q \qquad \mbox{for } s \ge 0 \qquad \mbox{and } 
\begin{cases}
q \in [0,\infty) &\mbox{for } d=2,\\
q \in [0,4) &\mbox{for } d=3.
\end{cases}
\end{align*}
Note that in \cite{NonlinearCN} the admissible growth condition in $3d$ requires $q \in [0,2)$, which is however a typo and should be, as above, $q \in [0,4)$ (cf. \cite[Proposition 3.2.5 and Remark 3.2.7]{Cazenave} for the original result).
Examples of nonlinearities that fulfill these assumptions are mentioned in the introduction. The most common and physically relevant choices covered by our setting are power law nonlinearities $\gamma(\rho)=\kappa \rho^{q}$ for $\kappa\ge 0$ and $0\le q<\infty$ in  $2d$ and $0\le q < 4$ in $3d$. Other physically relevant nonlinearities that fulfill the conditions are saturated nonlinearities appearing in the modeling of optical wave propagation such as $\gamma(\rho)=\kappa \rho (1+ \alpha \rho)^{-1}$ for $\alpha,\kappa\ge 0$. 

The above assumptions cover the regime of so-called defocussing (positive) nonlinearities and guarantees that the NLS and its Crank-Nicolson discretization are well-posed. For focussing (negative) nonlinearities, i.e. $\gamma : [0,\infty) \mapsto (-\infty,0]$, the well-posedness (of both the continuous and discrete models) can no longer be guaranteed without making additional technical assumptions. Typically, effects such as finite time blow ups can occur in this regime. To avoid constantly having to invoke a saving clause we restrict our attention to the defocussing case. We do however point out that under the assumptions that the NLS and the Crank-Nicolson discretizations are well-posed (without blow-up in the time interval $[0,T]$) then all our error estimates hold without changes.

For the initial value we assume that $u^0 \in H^1_0(\mathcal{D}) \cap H^2(\mathcal{D})$ and, without loss of generality, that it has a normalized mass, i.e. $\int_{\mathcal{D}} |u^0(x)|^2 \hspace{2pt}dx =1$. With this, the considered nonlinear Schr\"odinger equation (NLS) reads as follows. For a maximum time $T>0$ and an initial value $u^0$, we seek
$$
u \in L^{\infty}([0,T],H^1_0({\mathcal{D}})) 
\qquad
\mbox{and}
\qquad
\partial_t  u \in L^{\infty}([0,T],H^{-1}({\mathcal{D}}))
$$
such that $u(\cdot,0)=u^0 $ and
\begin{eqnarray}
\label{model-problem}\ci \partial_t u = - \Delta u 
+  V \hspace{1pt} u + \gamma( |u|^2 ) \hspace{1pt} u
\end{eqnarray}
 in the sense of distributions. Problem \eqref{model-problem} admits at least one solution, that is even unique for repulsive cubic nonlinearities in $1d$ and $2d$ (cf. \cite{Cazenave} in general and \cite[Remark 2.1]{NonlinearCN} for precise references). We assume that the solution admits the following additional regularity, which is
\begin{align}
\label{regularity-assumptions}
&u_{ttt} \in L^2(0,T;H^1(\mathcal{D})), \qquad  \qquad \partial_t^{(k)} u \in L^2(0,T;H^2(\mathcal{D})) \quad \mbox{for } 0\le k \le 2 \\
\mbox{and} \quad &u\in L^{\infty}(0,T;W^{1,\infty}(\mathcal{D})), \label{W1infty}
\end{align}
where we note that any solution with the increased regularity of \eqref{regularity-assumptions} must be unique (cf. \cite[Lemma 3.1]{NonlinearCN}).  In the rest of the paper $u$ hence always refers to this uniquely characterized solution.

It is well known that solutions to the NLS \eqref{model-problem} preserve the mass, i.e.
$$
\int_{\mathcal{D}} |u(t,x)|^2 \hspace{2pt} dx = \int_{\mathcal{D}} |u^0(x)|^2 \hspace{2pt} dx = 1
$$
and the energy, i.e.
$$
E[ u(t) ] = E[u^0], \qquad \mbox{where } E[u]:= \frac{1}{2} \int_{\mathcal{D}} |\nabla u(x)|^2 + V(x) \hspace{2pt} |u(x)|^2  + \Gamma(|u(x)|^2) \hspace{2pt} dx,
$$
with $\Gamma(\rho) := \int_0^{\rho} \gamma(r)\hspace{2pt}dr$.

For brevity, we shall denote the $L^2$-norm of a function $v \in L^2(\mathcal{D}):=L^2(\mathcal{D},\C)$ by $\| v \|$. The $L^2$-inner product is denoted by $\langle v, w  \rangle=\int_{\mathcal{D}} v(x) \hspace{2pt} \overline{w(x)} \hspace{2pt} dx$. Here, $\overline{w}$ denotes the complex conjugate of $w$.

Throughout the paper we will use the notation $A \lesssim B$, to abbreviate $A \le C B$, where $C$ is a constant that only depends on $u$, $T$, $d$, $\mathcal{D}$, $V$ and $\gamma$, but not on the discretization.

\begin{remark}In the analysis we restrict our attention to homogeneous Dirichlet boundary conditions. Typically these boundary conditions can be motivated by physical reasoning. For example in the context of Bose Einstein condensates, the magnetic potential $V$ is a trapping potential that becomes very quickly very large and hence traps the condensate in a bounded region. Mathematically this leads to an exponential decay of the solution $u$ to zero (in moderate distances from the origin of the coordinate system) and hence justifies to truncate the computational domain to a simple geometric object on which the problem is solved with zero boundary conditions. A typical alternative found in the literature are periodic boundary conditions which are e.g. favorable for spectral methods. Both the formulation of the Crank-Nicolson method and its error analysis can be easily generalized to that case. 
 \end{remark}
 
\section{Time-discrete Crank-Nicolson scheme} 
\label{section-time-discrete-CN}
In this section we will state the semi-discrete Crank-Nicolson scheme, recall its well-posedness and available stability bounds, and then use these results to prove optimal $L^{\infty}(H^1)$-error estimates in the Hilbert space setting. For that, let $T$ denote the final time of computation,  $N$ the number of time-steps, and $\tau = T/N$ the time step size. By $t_n$ we shall mean $t_n=n\tau$. The exact solution at time $t_n$ shall be denoted by $u^n:=u(t_n,\cdot)$. We also introduce a short hand notation for discrete time derivatives which is $D_\tau u^n := (u^{n+1}-u^n)/\tau $ and analogously $D_\tau \hatu{n} := ( \hatu{n+1} - \hatu{n})/\tau $.

\subsection{Method formulation and main result}
With the notation above, the semi-discrete Crank--Nicolson approximation $\hatu{n+1} \in H^1_0(\mathcal{D})$ to $u^{n+1}$ is given recursively as the solution (in the sense of distributions) to the equation
\begin{eqnarray}
\label{semi-disc-cnd-problem}
 \ci D_\tau \hatu{n}
 = - \Delta \hatu{n+\frac{1}{2}}
+  V \hspace{1pt} \hatu{n+\frac{1}{2}} + 
\frac{\Gamma(|\hatu{n+1}|^2)-\Gamma(|\hatu{n}|^2)}{|\hatu{n+1}|^2-|\hatu{n}|^2} \hspace{1pt} \hatu{n+\frac{1}{2}},
\end{eqnarray}
where $\hatu{n+\frac{1}{2}}:=(\hatu{n}+\hatu{n+1})/2$. The initial value is selected as $\hatu{0}=u^0$. It is easily seen that the discretization conserves both mass and energy, i.e.
\begin{align*}
\int_{\mathcal{D}} |\hatu{n}|^2 \hspace{2pt} dx =  \int_{\mathcal{D}} |u^0|^2 \hspace{2pt} dx \qquad
\mbox{and} \qquad E[\hatu{n}]  = E[u^0] \qquad \mbox{for all } n\ge 0.
\end{align*}
The scheme \eqref{semi-disc-cnd-problem} is well-posed and admits a set of a priori error estimates. The properties are summarized in the following theorem that is proved in \cite[Theorem 4.1]{NonlinearCN}.
\begin{theorem}
\label{main-theorem-1}
Under the general assumptions of this paper, there exists a constant $C(u)>0$ and a solution $\hatu{n}\in H^1_0(\mathcal{D})$ to the semi-discrete Crank-Nicolson scheme \eqref{semi-disc-cnd-problem} that is uniquely characterized by the property that
\begin{align}
\label{bounds-semi-discrete}
\sup_{0\le n \le N} \left( \| \hatu{n} \|_{L^{\infty}(\mathcal{D})} + \| \hatu{n} \|_{H^{2}(\mathcal{D})}\right) \le C(u),
\end{align}
and the a priori estimate for the $L^2$-error
\begin{align*}
\sup_{0\le n \le N} \| \hatu{n} - u^n \|
\lesssim \deltat{}^2,
\end{align*}
where $u$ is the (unique) exact solution with the regularity property \eqref{regularity-assumptions}.
\end{theorem}

Our main result on optimal error estimates in the $L^{\infty}(H^1)$ reads as follows.
\begin{theorem}[Optimal $H^1$-error estimates for the semi-discrete method] \label{MainResult}
Consider the setting of Theorem \ref{main-theorem-1}, then the $L^{\infty}(H^1)$-error converges with optimal order in $\tau$, i.e.
\begin{align*}
\sup_{0\le n \le N} \| \hatu{n} - u^n \|_{H^{1}(\mathcal{D})} \lesssim \deltat{}^2.
\end{align*}
\end{theorem}
The theorem is proved in Section \ref{proofs-semidiscrete-case} below.

\subsection{Proof of Theorem \ref{MainResult}}
\label{proofs-semidiscrete-case}

In this section we will prove Theorem \ref{MainResult}. Let us introduce some notation that is used throughout the proofs. We recall $D_\tau e^n = ( e^{n+1} - e^{n})/\tau $. Furthermore, we let $e^{n+1/2}:=(e^{n+1} + e^n)/2$ and $u^{n+1/2}:=(u^{n+1} + u^n)/2$. For time derivatives at fixed time $t^n$, we also write $\partial_t u^n:= \partial_t u(t^n,\cdot)$.

We begin by establishing a differential equation for the time discrete error $e^n = u^n-u^n_\tau$. This is stated in the following lemma.
\begin{lemma}[Consistency error]\label{error_pde-lemma}
The error $e^n = u^n-u_\tau^n\ $  fulfills the identity 
\begin{equation} \label{error_pde}
\ci D_\tau e^n+\Delta e^{n+1/2}-Ve^{n+1/2}-e^n_\gamma= T^n,
\end{equation}
where the consistency error $T^n$ is given by
\begin{eqnarray}\label{Taylor}
\lefteqn{T^n := \ci \hspace{2pt}( D_\tau u^n - \partial_t u(t_{n+1/2}))+\Delta( u^{n+1/2}-u(t_{n+1/2})) -V(u^{n+1/2}-u(t_{n+1/2}))} \nonumber \\
&\enspace& -( \gamma(\xi^n)  u^{n+1/2} -\gamma(|u(t_{n+1/2})|^2) u(t_{n+1/2}) ).\hspace{250pt}
  \end{eqnarray} 
Here, $e^n_\gamma := \gamma(\xi^n)u^{n+1/2}-\gamma(\xi^n_\tau)u^{n+1/2}_\tau$ for some bounded functions
$\xi^n,\xi^n_\tau  \in L^{\infty}(\mathcal{D})$ with the properties that 
\begin{align*}
\xi^n(x) &\in [\min(|u^n|^2,|u^{n+1}|^2),\max(|u^n|^2,|u^{n+1}|^2)] \qquad \mbox{and} \\
\xi^n_\tau(x) &\in [\min(|u^n_\tau|^2,|u^{n+1}_\tau|^2),\max(|u^n_\tau|^2,|u^{n+1}_\tau|^2)]
\end{align*} 
for almost all $x\in \mathcal{D}$.
\end{lemma}
\begin{proof}
It is easily verified that exact solution fulfills
\begin{equation} \label{ExactDiscrete}
\ci D_\tau u^n + \Delta u^{n+1/2}- Vu^{n+1/2} -\frac{\Gamma(|u^{n+1}|^2)-\Gamma(|u^n|^2)}{|u^{n+1}|^2-|u^n|^2} u^{n+1/2}  = T^n.
\end{equation}
By the regularity assumptions we can apply Taylor expansion arguments to $T^n$ to see:
\begin{align}
\label{bounds-for-Tk}
\sum_{k=0}^{N}\|T^k\|^2 \leq C\tau^3
 \end{align}
The argument that proves \eqref{bounds-for-Tk} is elaborated in Appendix \ref{consistency-error-appendix}, where it also becomes visible how the regularity assumptions enter explicitly in the estimate.
Next, subtracting \eqref{ExactDiscrete} from \eqref{semi-disc-cnd-problem} we find that $e^n = u^n-u_\tau^n$ satisfies: 
\[\ci D_\tau e^n+\Delta e^{n+1/2}-Ve^{n+1/2}-e^n_\gamma= T^n \]
where $e^n_\gamma$ denotes the error coming from the nonlinear term, defined by \[e_\gamma^n=  \frac{\Gamma(|u^{n+1}|^2)-\Gamma(|u^n|^2)}{|u^{n+1}|^2-|u^n|^2}u^{n+1/2}-\frac{\Gamma(|u_\tau^{n+1}|^2)-\Gamma(|u^n_\tau|^2)}{|u^{n+1}_\tau|^2-|u^n_\tau|^2} u_\tau^{n+1/2}.  \]
Recalling the definition of $\Gamma$ we have:
 \begin{equation*}
 \frac{\Gamma(|u^{n+1}|^2)-\Gamma(|u^n|^2)}{|u^{n+1}|^2-|u^n|^2} = \frac{1}{|u^{n+1}|^2-|u^n|^2}\int_{|u^n|^2}^{|u^{n+1}|^2} \gamma(r)\hspace{2pt}dr =\colon\gamma(\xi^n) , 
\end{equation*}
likewise
\begin{align*}
\frac{\Gamma(|u^{n+1}_\tau|^2)-\Gamma(|u^n_\tau|^2)}{|u^{n+1}_\tau|^2-|u^n_\tau|^2} = \frac{1}{|u^{n+1}_\tau|^2-|u^n_\tau|^2}\int_{|u^n_\tau|^2}^{|u^{n+1}_\tau|^2} \gamma(r)\hspace{2pt}dr =: \gamma(\xi^n_\tau).
\end{align*}
The expression for $e^n_\gamma$ is thus simplified to
\begin{align*}
e^n_\gamma := \gamma(\xi^n)u^{n+1/2}-\gamma(\xi^n_\tau)u^{n+1/2}_\tau,
\end{align*}
where $\xi^n$ is a function taking values between $|u^n|^2$ and $|u^{n+1}|^2$ and $\xi^n_\tau$ a function taking values between $|u^n_\tau|^2$ and $|\hatu{n+1}|^2$.
\end{proof}
The differential equation in Lemma \ref{error_pde-lemma} is now used to derive a  recurrence formula for the $H^1$-norm of the error. Multiplying \eqref{error_pde} by $D_\tau e^n$, integrating and taking the real part yields:
\begin{align} \label{H1-Reccurence}
\frac{\|\nabla e^{n+1}\|^2-\|\nabla e^n\|^2}{2\tau} = \underbrace{ \text{Re}( \langle e_\gamma^n, D_\tau e^n\rangle)}_{\mbox{\rm I}} \hspace{5pt}\underbrace{-\hspace{3pt}\text{Re}(\langle T^n,D_\tau e^n\rangle
).}_{\mbox{\rm II}}
\end{align}
The idea is to bound the terms I and II in such a way that Grönwall's inequality can be used. 
We proceed to bound term I. Multiplying the error PDE \eqref{error_pde} by $e^n_\gamma$ results in:
\begin{align*}
\ci \langle D_\tau e^n, e_\gamma^n\rangle = \langle \nabla e^{n+1/2},\nabla e_\gamma^n\rangle + \langle Ve^{n+1/2},e^n_\gamma\rangle +  \| e_\gamma^n\|^2 + \langle T^n,e_\gamma^n\rangle  
\end{align*}
and consequently
\begin{eqnarray} \label{TermI}
\nonumber\lefteqn{|\mbox{\rm I}| = |\text{Re}( \langle D_\tau e^n, e_\gamma^n\rangle)| }\\
\nonumber&\leq& |\text{Im}(\langle \nabla e^{n+1/2}, \nabla e_\gamma^n)\rangle |+|\text{Im}\langle Ve^{n+1/2},e^n_\gamma\rangle|+ |\text{Im}(\langle T^n, e_\gamma^n\rangle )| \nonumber \\ 
&\lesssim& \|\nabla e^{n+1/2} \|^2+ \|\nabla e^n_\gamma\|^2+ \|V\|_\infty(\|e^{n+1/2}\|^2+ \|e^n_\gamma\|^2)+\|e^n_\gamma\|^2+\|T^n\|^2 \nonumber \\
&\lesssim& \|\nabla e^{n+1}\|^2+\|\nabla e^{n}\|^2+\|\nabla e^n_\gamma\|^2+\|e^n_\gamma\|^2+\|T^n\|^2+\tau^4.
\end{eqnarray}
In order to use Gr\"onwall's inequality we need to bound $\|e^n_\gamma\|$ and $\|\nabla e^n_\gamma\|$ in terms of $\|e^n\|$,$\|\nabla e^n\|$ and terms of $\mathcal{O}(\tau^2)$. These bounds are formulated in the two following lemmas.  

\begin{lemma}
\label{lemma-3.4}
Given the optimal $L^2$-convergence of Theorem \ref{main-theorem-1} and the uniform bounds \eqref{bounds-semi-discrete}, the error coming from the nonlinear term behaves as $\tau^2$, i.e. $\|e^n_\gamma\| \lesssim \tau^2 $. 
\begin{proof}
	We introduce the function $f$ to denote how $\gamma(\xi^n)$ depends on $|u^n|^2$ and $|u^{n+1}|^2$
	\begin{eqnarray*}
	f(a,b) = \frac{1}{b-a}\int_a^b\gamma(r)dr.
\end{eqnarray*}
The derivative of $f$ with respect to to its $i$-th input is denoted $f_i$, i.e. $ f_1(a,b)=\partial_a f(a,b)$. A standard application of the mean value theorem yields that, for some function $ \eta^n $ taking values between $|u^n|^2$ and $|u^n_\tau|^2$ a.e. and for some function $\eta^{n+1} $ likewise between $|u^{n+1}|^2$ and $|u^{n+1}_\tau|^2$, it holds:
\begin{eqnarray*}
	\lefteqn{e^n_\gamma = (\gamma(\xi^n)-\gamma(\xi^n_\tau))u^{n+1/2} + \gamma(\xi^n_\tau)(u^{n+1/2}-u^{n+1/2}_\tau)} \\ 
	&=& (f(|u^n|^2,|u^{n+1}|^2)-f(|u^n_\tau|^2,|u^{n+1}_\tau|^2) ) u^{n+1/2}+\gamma(\xi^n_\tau)e^{n+1/2} \\
	&=& \big(f_1(\eta^n,\eta^{n+1})(|u^n|^2-|u^n_\tau|^2)+f_2(\eta^n,\eta^{n+1})(|u^{n+1}|^2-|u^{n+1}_\tau|^2)  \big)u^{n+1/2}+\gamma(\xi^n_\tau)e^{n+1/2}.
\end{eqnarray*}
 A quick sanity check shows that $f_1$ and $f_2$ are bounded by the derivative of $\gamma$:
\begin{eqnarray*}
	|f_1(a,b)|&=& | \frac{1}{b-a}(\gamma(c)-\gamma(a))| = |\gamma'(\theta^-)\frac{c-a}{b-a}|\\
	|f_2(a,b)|&=& |\frac{1}{b-a}(\gamma(b)-\gamma(c)) | =| \gamma'(\theta^+)\frac{b-c}{b-a}| ,
\end{eqnarray*}
where $c, \theta^-$ and $\theta^+$ lie somewhere between $a$ and $b$.
With the $L^\infty$-bounds on $u^n$ and $u^n_\tau$ it is now straightforward to show $\|e^n_\gamma\|\leq C_{\gamma'}( \|e^n\|+\|e^{n+1}\|)$:
\begin{eqnarray*}
	\lefteqn{\|e^n_\gamma\| \leq  \|f_1(\eta^n,\eta^{n+1})\|_{L^\infty} \|u^{n+1/2}\|_{L^\infty}\| |u^n|^2-|u^n_\tau|^2\|} \\
	& & \qquad + \|f_2(\eta^n,\eta^{n+1})\|_{L^\infty} \|u^{n+1/2}\|_{L^\infty}\| |u^{n+1}|^2-|u^{n+1}_\tau|^2\| \\
	& & \qquad + \|\gamma(\xi^n_\tau)\|_{L^\infty} \|e^{n+1/2}\|.
\end{eqnarray*}
As $\| |u^n|^2-|u^n_\tau|^2\| \leq \| |u^n|+|u^{n+1}|\|_{L^\infty} \|e^n\|$, it now follows that 
\[\|e^n_\gamma\| \lesssim \|e^{n+1}\|+\|e^n\| \lesssim \tau^2, \]
and the lemma is proved.
 \end{proof}
\end{lemma}

\begin{lemma}
\label{lemma-3.5}
Given Theorem \ref{main-theorem-1}, the gradient of the error coming from the nonlinear term is bounded as $\|\nabla e^n_\gamma\| \lesssim \| \nabla e^{n+1}  \|+\|\nabla e^n\|+\tau^2.$
\begin{proof}
The steps are much the same as in lemma \eqref{lemma-3.4}, with the exception that we need to use $ W^{1,\infty}$-bounds on $u^n$ \eqref{W1infty}, which are not available for $u^n_\tau$. We begin by splitting $\nabla e^n_\gamma$ into terms so that the previous lemma may be used.
	\begin{eqnarray*}
		\lefteqn{\nabla e^n_\gamma = \nabla [\gamma(\xi^n)u^{n+1/2}-\gamma(\xi^n_\tau) u^{n+1/2}_\tau] } \\
		&=&\nabla \gamma(\xi^n)u^{n+1/2}+\gamma(\xi^n)\nabla u^{n+1/2} -\nabla \gamma(\xi^n_\tau)u^{n+1/2}_\tau-\gamma(\xi^n_\tau)\nabla u^{n+1/2}_\tau \\
		&=& u^{n+1/2}_\tau \nabla (\gamma(\xi^n)-\gamma(\xi^n_\tau))+\nabla \gamma(\xi^n)(u^{n+1/2}-u^{n+1/2}_\tau) + (\gamma(\xi^n)-\gamma(\xi^n_\tau))\nabla u^{n+1/2} \\
		&\enspace& \quad + \gamma(\xi^n_\tau)\nabla(u^{n+1/2}-u^{n+1/2}_\tau)
	\end{eqnarray*}
By the previous lemma and the $W^{1,\infty}$-bound on $u^n$ we may conclude, \begin{eqnarray} \label{partial_result}
	\|\nabla e^n_\gamma\| \leq \|u^{n+1/2}_\tau\|_{L^\infty} \|\nabla (\gamma(\xi^n)-\gamma(\xi^n))\| + \|\nabla \gamma(\xi^n)\|_{L^\infty}\|e^{n+1/2}\|\\ 
	+\|\nabla u^{n+1/2}\|_{L^\infty}\tau^2 +\|\gamma(\xi^n_\tau)\|_{L^\infty}\|\nabla e^{n+1/2}\| \nonumber.
\end{eqnarray}
	What is left to bound is the term $\|\nabla (\gamma(\xi^n)-\gamma(\xi^n_\tau))\|$. We consider its dependence on $|u^n|^2,|u^{n+1}|^2,|u^n_\tau|^2$ and $|u^{n+1}_\tau|^2$:
	\begin{eqnarray} \label{second_lemma}
	\lefteqn{\nabla(\gamma(\xi^n)-\gamma(\xi^n_\tau)) }\nonumber \\
	&=& 	\nabla(f(|u^n|^2,|u^{n+1}|^2)-f(|u^n_\tau|^2,|u^{n+1}_\tau|^2)) \nonumber \\
	&=& f_1\nabla|u^n|^2+f_2\nabla|u^{n+1}|^2-f_{1,\tau}\nabla|u^n_\tau|^2-f_{2,\tau}\nabla|u^{n+1}_\tau|^2 \nonumber \\
	&=&	(f_1-f_{1,\tau})\nabla|u^n|^2 + f_{1,\tau}\nabla(|u^n|^2-|u^n_\tau|^2)  \\
	&& + (f_2-f_{2,\tau})\nabla|u^{n+1}|^2+f_{2,\tau}\nabla(|u^{n+1}|^2-|u^{n+1}_\tau|^2). \nonumber
	\end{eqnarray}
	Where $f_1$ is to be read as $f_1(|u^n|^2,|u^{n+1}|^2)$ and $f_{1,\tau}$ as $f_1(|u^n_\tau|^2,|u^{n+1}_\tau|^2)$ and likewise for $f_2$ and $f_{2,\tau}$. Another application of the mean value theorem yields:
	\begin{eqnarray*}
		f_1-f_{1,\tau} = f_{1,1}(\theta^n,\theta^{n+1})(|u^n|^2-|u^n_\tau|^2) + f_{1,2}(\theta^n,\theta^{n+1})(|u^{n+1}|^2-|u^{n+1}_\tau|^2)  \\
			f_2-f_{2,\tau} = f_{2,1}(\vartheta^n,\vartheta^{n+1})(|u^n|^2-|u^n_\tau|^2) + f_{2,2}(\vartheta^n,\vartheta^{n+1})(|u^{n+1}|^2-|u^{n+1}_\tau|^2) ,
	\end{eqnarray*}
for some $\theta^n$, $\vartheta^n$ between $|u^n|^2$ and $|u^n_\tau|^2$ and some $\theta^{n+1}$, $\vartheta^{n+1}$ between $|u^{n+1}|^2$ and $|u^{n+1}_\tau|^2$. The following quick calculations show that the partial derivatives of $f$ of order two are bounded by $\gamma''$.
	\begin{eqnarray*}
		f_{1,1}(a,b) &=& \frac{1}{(b-a)^2}2(\gamma(c)-\gamma(a))-\frac{\gamma'(a)}{b-a}= \frac{\gamma'(\theta^-)-\gamma'(a)}{b-a}+C_{\gamma''}\\
		f_{1,2}(a,b) &=& \frac{1}{(b-a)^2}(\gamma(b)-2\gamma(c)+\gamma(a)) = \gamma''(c)+C_{\gamma''}\\
		f_{2,2}(a,b) &=& \frac{\gamma'(b)}{b-a}-\frac{1}{(b-a)^2}2(\gamma(b)-\gamma(c))= \frac{\gamma'(b)-\gamma'(\theta^+)}{b-a}+C_{\gamma''}
	\end{eqnarray*}
Where it was used that $c = (b+a)/2+ C_{\gamma''}(a-b)^2$. 
It thus becomes clear that $\|f_{1,1}+ f_{1,2}+f_{2,2}\|_{L^\infty} \leq C_{\gamma''}$. This gives us the following $L^2$-bound on \eqref{second_lemma} 
\begin{eqnarray*}
	\|\nabla (\gamma(\xi^n)-\gamma(\xi^n_\tau))\| \leq C_{\gamma''}(\|e^n\|+\|e^{n+1}\|)+C_{\gamma'}\nabla(|u^{n}|^2-|u^{n}_\tau|^2+|u^{n+1}|^2-|u^{n+1}_\tau|^2).
\end{eqnarray*}
Continuing from \eqref{partial_result}, we now conclude that:
	\begin{eqnarray} \label{last_step}
		\|\nabla e^n_\gamma \| \lesssim \|\nabla(|u^n|^2-|u^n_\tau|^2)\|+\|\nabla(|u^{n+1}|^2-|u^{n+1}_\tau|^2)\|+\tau^2.
		\end{eqnarray}
It is noted that $\nabla(|u^n|^2-|u^n_\tau|^2)$ may be written
\begin{eqnarray*}
\nabla(|u^n|^2-|u^n_\tau|^2) = 2\text{\normalfont Re}( (u^n-u^n_\tau)\nabla \overline{u}^n+u^n_\tau \nabla(\overline{u}^n-\overline{u}^n_\tau)).
\end{eqnarray*}
Using the $W^{1,\infty}$-bound available for $u^n$ we have that $\|\nabla(|u^n|^2-|u^n_\tau|^2)\|\lesssim \|\nabla (u^n-u^n_\tau)\| $. With this, eq. \eqref{last_step} becomes $\|\nabla e^n_\gamma\| \lesssim \|\nabla e^n\|+\|\nabla e^{n+1}\|+\tau^2$ and the lemma is proved.
\end{proof}
\end{lemma}
With Lemma \ref{lemma-3.4} and \ref{lemma-3.5} we now have the following bound on term I.
\begin{lemma}
\label{estimate-term-I}
For term $\mbox{\rm I}$ which is given by \eqref{TermI}, we have the estimate
\begin{equation} \label{BndI}
|\mbox{\rm I}|\hspace{3pt}
\lesssim \hspace{3pt} \|\nabla e^{n+1} \|^2+\|\nabla e^n \|^2 + \|T^n\|^2+\tau^4.
\end{equation}
\end{lemma}
We can now proceed to bound term II. Here we explicate the Taylor term using \eqref{Taylor} to see
\begin{eqnarray*}
\lefteqn{\mbox{\rm II} = - \text{Re}(\langle T^n , D_\tau e^n \rangle)}  \\
 &\leq& \underbrace{ | \langle( D_\tau u^n - \partial_t u(t_{n+1/2})),D_\tau e^n\rangle|}_{\mbox{\rm IIa}}+\underbrace{\text{Re}\langle -\Delta( u^{n+1/2}-u(t_{n+1/2})),D_\tau e^n\rangle}_{\mbox{\rm IIb}} \\
&\enspace& +\underbrace{|\langle V(u^{n+1/2}-u(t_{n+1/2})),D_\tau e^n\rangle|}_{\mbox{\rm IIc}} \\
&\enspace&+\underbrace{|\langle \bigg( \frac{\Gamma(|u^{n+1}|^2)-\Gamma(|u^n|^2)}{|u^{n+1}|^2-|u^n|^2}-\gamma(|u(t_{n+1/2})|^2),D_\tau e^n\rangle \bigg)|}_{\mbox{\rm IId}}.
\end{eqnarray*}
We start with estimating $\mbox{\rm IIa},\mbox{\rm IIc}$ and $\mbox{\rm IId}$, which can be bounded in a similar way. \\
\it Step 1, bounding $\mbox{\rm IIa}$: \\ \normalfont
By replacing $D_\tau e^n$ using \eqref{error_pde}  (i.e. time derivative is replaced by regularity in space) we have
\begin{eqnarray*}
\lefteqn{|\langle D_\tau u^n-\partial_t u(t_{n+1/2}) , D_\tau e^n \ \rangle |}\\
&\leq&  |\langle D_\tau u^n-\partial_t u(t_{n+1/2}) ,\Delta e^{n+1/2}\rangle|+|\langle D_\tau u^n-\partial_t u(t_{n+1/2}) ,V e^{n+1/2}\rangle|  \\ 
&\enspace& \quad + \enspace | \langle D_\tau u^n-\partial_t u(t_{n+1/2}) , e_\gamma^n\rangle | +  | \langle D_\tau u^n-\partial_t u(t_{n+1/2}) ,T^n \rangle | \\
&\overset{\eqref{bounds-for-Tk}}{\lesssim}& \|\nabla e^{n+1}\|^2 + \|\nabla e^n\|^2 + \|D_\tau u^n-\partial_t u(t_{n+1/2})\|_{H^1}^2 + \|T^n\|^2 +\tau^4 .
\end{eqnarray*}
 The term $\|D_\tau u^n-\partial_t u(t_{n+1/2})\|_{L^2}^2$ was absorbed in $\lesssim \|T^n\|^2 $. Here we see how the assumption that $\partial_{ttt}u\in L^2(0,T;H^1(\mathcal{D}))$ will enter, as it will allow us to conclude that $\sum \|D_\tau u^n-\partial_t u(t_{n+1/2})\|_{H^1}^2\lesssim \tau^3$. \\
\it Step 2, bounding $\mbox{\rm IIc}$: \\ \normalfont
We use the same idea as for $\mbox{\rm IIa}$ to get
\begin{eqnarray*}
\lefteqn{|\langle V(u^{n+1/2}-u(t_{n+1/2})),D_\tau e^n\rangle| \leq \|V\|_{L^\infty}|\langle u^{n+1/2}-u(t_{n+1/2}),D_\tau e^n\rangle|} \\
&\leq& \|V\|_{L^\infty}\big( |\langle u^{n+1/2}-u(t_{n+1/2}) ,\Delta e^{n+1/2}\rangle|+\\
&\enspace& \qquad +|\langle u^{n+1/2}-u(t_{n+1/2}) ,V e^{n+1/2}\rangle| +|   \langle u^{n+1/2}-u(t_{n+1/2}) , e_\gamma^n\rangle | \\ 
&\enspace& \qquad +   | \langle u^{n+1/2}-u(t_{n+1/2}) ,T^n \rangle | \big)  \\
&\lesssim& \|\nabla e^{n+1}\|^2 + \|\nabla e^n\|^2 + \|u^{n+1/2}-u(t_{n+1/2})\|_{H^1}^2+ \|T^n\|^2+\tau^4.
\end{eqnarray*}
\it Step 3, bounding \mbox{\rm IId}: \\ 
\normalfont
We start from
\begin{align*}
\big|\big\langle\frac{\Gamma(|u^{n+1}|^2)-\Gamma(|u^n|^2)}{|u^{n+1}|^2-|u^n|^2}-\gamma(|u(t_{n+1/2})|^2),D_\tau e^n\big\rangle \big| = & |\langle \gamma(\xi^n)-\gamma(|u(t_{n+1/2})|^2), D_\tau e^n\rangle|
\end{align*}
and replace $D_\tau e^n$ again using \eqref{error_pde}. Furthermore, in virtue of the assumptions  it holds that $\|\nabla(\gamma(\xi^n)-\gamma(|u(t_{n+1/2})|^2))\| \lesssim \|u^{n+1/2}-u(t_{n+1/2})\|_{H^1}$, this is made explicit in the Appendix \eqref{consistency-error-appendix}.   We thus obtain 
\begin{eqnarray*}
\lefteqn{|\langle \gamma(\xi^n)-\gamma(|u(t_{n+1/2})|^2), D_\tau e^n\rangle|} \\
&=& |\langle \gamma(\xi^n)-\gamma(|u(t_{n+1/2})|^2), -\Delta e^{n+1/2} + Ve^{n+1/2}+e^n_\gamma +T^n \rangle|  \\
&\leq& |\langle \nabla(\gamma(\xi^n)-\gamma(|u(t_{n+1/2})|^2)),\nabla e^{n+1/2}\rangle | +\|T^n\|^2 + \mathcal{O}(\tau^4) \\
&\lesssim & \|\nabla e^{n+1}\|^2+\|\nabla e^n\|^2+ \|T^n\|^2+\|u^{n+1/2}-u(t_{n+1/2})\|_{H^1}^2+ \tau^4.
\end{eqnarray*}
\it Step 4, bounding $\mbox{\rm IIb}$: \\ \normalfont 
The previous technique does not work on this term since replacing the discrete time derivative with regularity in space would give rise to the term $\nabla \Delta (u^{n+1/2}-u(t_{n+1/2}))$, which we can not afford. Instead we use summation by parts in time to get the factor $D_\tau \Delta (u^{n+1/2}-u(t_{n+1/2}))$, which when integrated against $e^{n+1/2}$ can be handled.
First we recall:
\[D_\tau[a^k b^k] =  a^kD_\tau b^k + b^{k+1}D_\tau a^k \quad \Longleftrightarrow \quad -a^kD_\tau b^k = D_\tau [a^k] b^{k+1} - D_\tau[a^kb^k]. \]
Using this on term $\mbox{\rm IIb}$ yields:
\begin{eqnarray*}
\lefteqn{\langle -\Delta (u^{n+1/2}-u), D_\tau e^n\rangle  = \langle \nabla (u^{n+1/2}-u),D_\tau \nabla e^n \rangle} \\
& =& D_\tau [\langle \nabla (u^{n+1/2}-u(t_{n+1/2})), \nabla e^{n}\rangle ] - \langle D_\tau \nabla(u^{n+1/2}-u(t_{n+1/2})),\nabla e^{n+1}\rangle \\
& \leq & D_\tau [\langle \nabla (u^{n+1/2}-u(t_{n+1/2})), \nabla e^{n}\rangle ] +  |\langle D_\tau \nabla(u^{n+1/2}-u(t_{n+1/2})),\nabla e^{n+1}\rangle| \\
&\leq& D_\tau [\langle \nabla (u^{n+1/2}-u(t_{n+1/2})), \nabla e^{n}\rangle ] + \|D_\tau (u^{n+1/2}-u(t_{n+1/2}))\|_{H^1}^2 + \|\nabla e ^{n+1}\|^2.
\end{eqnarray*} 
Collecting the estimates we have the following estimate for term $\mbox{\rm II}$.
\begin{lemma}
\label{estimate-term-II}
For term $\mbox{\rm II} = -\text{Re}(\langle T^n , D_\tau e^n \rangle)$ it holds the estimate
\begin{eqnarray} \label{BndII}
\mbox{\rm II}  \hspace{3pt}
\le \hspace{3pt}  
D_\tau [\langle \nabla (u^{n+1/2}-u(t_{n+1/2})), \nabla e^{n}\rangle ] + C \big( \|\nabla e ^{n+1}\|^2 +\|\nabla e^n\|^2 + \tau^4 +\|T^n\|^2+ \nonumber \\
 + \|D_\tau u^n-\partial_t u(t_{n+1/2})\|_{H^1}^2 +\|u^{n+1/2}-u(t_{n+1/2})\|_{H^1}^2\big). \nonumber \\
\end{eqnarray}
\end{lemma}

\noindent 
Here we note the importance of not estimating the absolute value of the first term since it is necessary to use the fact that $n$ of these terms cancel when summed up, i.e. $\sum_k D_\tau a^k = \frac{1}{\tau}(a^{n+1}-a^0)$. We are now ready to finish the proof of the first main result.
\begin{proof}[Proof of Theorem \ref{MainResult}]
We pick off where we left (\ref{H1-Reccurence}) and find by using Lemma \ref{estimate-term-I} and \ref{estimate-term-II}:
\begin{align*}
\frac{\|\nabla e^{n+1}\|^2 -\|\nabla e^n\|^2}{2\tau} & \leq   D_\tau [\langle \nabla (u^{n+1/2}-u(t_{n+1/2})), \nabla e^{n}\rangle ] +C\big(\|\nabla e^{n+1}\|^2 +\|\nabla e^n\|^2 +\tau^4   \\
&  + \|D_\tau u^n -\partial_t u(t_{n+1/2})\|^2_{H^1}+ \|T^n\|^2  + \|u^{n+1/2}-u(t_{n+1/2})\|^2_{H^1} \big).
\end{align*}
Summing this up and using $e^0=0$ gives
\begin{align*}
\frac{\|\nabla e^{n+1}\|^2}{2\tau} &\leq C\left( \sum_{k=0}^n \|\nabla e^{k}\|^2 \right)  + \frac{1}{\tau}\langle \nabla (u^{n+3/2}-u(t_{n+3/2})),\nabla e^{n+1} \rangle + \\
&+ C\tau^3 + \sum_{k=0}^{n} \|T^k\|^2+\|u^{k+1/2}-u(t_{k+1/2}) \|^2_{H^1}+ \|D_\tau u^k -\partial_t u(t_{k+1/2})\|^2_{H^1}
\end{align*} 
and therefore, recalling \eqref{bounds-for-Tk},
\begin{align*}
\|\nabla e^{n+1}\|^2 &\leq C\left(\sum_{k=0}^n \tau \|\nabla e^{k}\|^2 \right)+ C \tau^4 +|\langle \nabla (u^{n+3/2}-u(t_{n+3/2})),\nabla e^{n+1} \rangle |.
\end{align*} 
Young's inequality with $\epsilon>0$ is used on the last term:
\begin{equation}
|\langle \nabla (u^{n+3/2}-u(t_{n+3/2})),\nabla e^{n+1} \rangle | \leq C( \frac{\tau^4}{\epsilon} + \epsilon \|\nabla e^{n+1}\|^2).
\end{equation}
Which holds since, 
\begin{align*}
\| \nabla (u^{n+3/2}-u(t_{n+3/2}))\| \lesssim \tau^2 \|\partial_{tt}u\|_{L^\infty(H^1)},
\end{align*}
where we have $ \|\partial_{tt}u\|_{L^\infty(H^1)} \lesssim \|\partial_{tt}u\|_{L^2(H^1)} + \|\partial_{ttt}u\|_{L^2(H^1)}$ by Sobolev embeddings.
Finally we arrive at
\begin{align*} 
\|\nabla e^{n+1}\|^2 \leq C\left( \sum_{k=0}^n \tau \|\nabla e^k\|^2 \right) + C \tau ^4 + \frac{\tau^4}{\epsilon} + \epsilon \|\nabla e^{n+1}\|^2 
\end{align*}
and for e.g. $\epsilon=1/2$ we can absorb $ \epsilon \|\nabla e^{n+1}\|^2$ in the left hand side and conclude
\begin{align*}
\|\nabla e^{n+1}\|^2 & \leq C \big( \tau^4 + \sum_{k=0}^n \tau \|\nabla e^k\|^2 \big).
\end{align*}
Gr\"onwall's inequality now yields:
\begin{equation} \label{Convergence}
\|\nabla e^{n+1}\| \lesssim  \tau^2.
\end{equation}
\end{proof}
\section{Fully-discrete Crank-Nicolson scheme}
\label{section-fully-discrete-CN}
We shall now consider the fully-discrete setting that is based on a finite element discretization in space.
For that, we let $S_h\subset H^1_0(\mathcal{D})$ denote the space of P1 Lagrange finite elements on a quasi-uniform simplicial mesh on $\mathcal{D}$ with mesh size $h$. In this setting we have by standard finite element theory (cf. \cite{BrS08}) the following estimate for any $u \in H^1_0(\mathcal{D})\cap H^2(\mathcal{D})$:
\begin{align}\label{ProjectionIneq}
\|u-P_h u\|+h\|\nabla (u-P_h u) \| \lesssim h^2\|u\|_{H^2(\mathcal{D})}.
\end{align} 
Here $P_h: H^1_0(\mathcal{D}) \rightarrow S_h$ denotes (for example) the Ritz-projection into the finite element space. 
For a given discrete initial value $u_h^0\in S_h$ the CN-FEM approximation $u^{n}_{\tau,h} \in S_h$ to $u^{n}$ is given by the fully discrete equation
	\begin{align}\label{FullyDiscrete}
	\ci \big\langle \frac{u^{n+1}_{\tau,h}-u^n_{\tau,h}}{\tau},v\big\rangle = 
	\big\langle\nabla u^{n+1/2}_{\tau,h},\nabla v \big\rangle + \big\langle Vu^{n+1/2}_{\tau,h},v\big\rangle + \left\langle \frac{\Gamma(|u^{n+1}_{\tau,h}|^2)-\Gamma(|u^{n}_{\tau,h}|^2)}{|u^{n+1}_{\tau,h}|^2-|u^n_{\tau,h}|^2}u^{n+1/2}_{\tau,h},v\right\rangle
	\end{align}
for all $v\in S_h$.
The initial value is selected as $ u^0_{\tau,h}=P_hu^0$. As in the semi-discrete case, the discretization conserves both mass and energy, i.e.
\begin{align*}
\int_{\mathcal{D}} | u^n_{\tau,h}|^2 \hspace{2pt} dx =  \int_{\mathcal{D}} |u^0_{\tau,h}|^2 \hspace{2pt} dx \qquad
\mbox{and} \qquad E[u^n_{\tau,h}]  = E[u^0_{\tau,h}] \qquad \mbox{for all } n\ge 0.
\end{align*}

The scheme is well-posed and the corresponding approximations converge in the $L^{\infty}(L^2)$-norm with optimal order in space and time to the exact solution. A proof of this statement can be easily extracted from \cite[Theorem 3.1 and Lemma 5.3]{NonlinearCN}. In particular, we have the following result.

\begin{theorem}\label{main-theorem-2}
Under the general assumptions of this paper, there exists a  solution $u^n_{\tau,h}\in H^1_0(\mathcal{D})$ to the fully discrete Crank-Nicolson scheme \eqref{FullyDiscrete} such that the following a priori error estimates hold
\begin{align*}
\sup_{0\le n \le N} \| u^n_{\tau,h} - u^n_\tau \|_{L^{2}(\mathcal{D})}  \lesssim h^2
\qquad \mbox{and} \qquad
\sup_{0\le n \le N} \| u^n_{\tau,h} - u^n \|_{L^{2}(\mathcal{D})}  \lesssim  h^2+\tau^2.
\end{align*}
\end{theorem}
With this we are ready to state our final theorem.
\begin{theorem}[Optimal $H^1$-error estimates for the fully discrete method] \label{MainResult2}
Let $\hatu{n}\in H^1_0(\mathcal{D})$ denote the fully-discrete Crank-Nicolson approximations from Theorem \ref{main-theorem-2}, then it holds
\begin{align*}
\sup_{0\le n \le N} \| u^n_{\tau,h} - u^n \|_{H^{1}(\mathcal{D})} \lesssim \deltat{}^2+h.
\end{align*}
\end{theorem}
\begin{proof}
First, we recall the inverse estimate on quasi-uniform meshes (cf. \cite{BrS08}), i.e. $\| \nabla v_h \| \le C h^{-1} \| v_h \|$ for all $v_h \in S_h$, which implies
\begin{align}
\label{inverse-estimate} \| \nabla(P_h(u_\tau^n)-u_{\tau,h}^n) \| \le C h^{-1} \| P_h(u_\tau^n)-u_{\tau,h}^n \|.
\end{align}
With this, the $H^1$ convergence result \eqref{Convergence} together with Theorem \ref{main-theorem-2}  suffice to show optimal $H^1$-convergence rates for the fully discrete method. This is made clear by the following splitting.
\begin{eqnarray*}
\|\nabla (u^n-u_{\tau,h}^n)\| &\leq& \|\nabla(u^n-u_\tau^n)\|+\|\nabla(u_\tau^n-P_h(u_\tau^n))\|+\|\nabla(P_h(u_\tau^n)-u_{\tau,h}^n)\| \\
&\overset{\eqref{inverse-estimate}}{\leq}& \|\nabla(u^n-u_\tau^n)\| + Ch+Ch.
\end{eqnarray*}
Here we have made use of the inequality \eqref{ProjectionIneq}, the uniform $H^2$-regularity of $u^n_\tau$, i.e. $\|u^n_\tau\|_{H^2} \leq C(u)$ (cf. \eqref{bounds-semi-discrete}) and the optimal $L^2$-estimates. In virtue of Theorem  \ref{MainResult} we may thus conclude:
\begin{equation}
\| \nabla (u^n-u^n_{\tau,h}) \| \leq C(\tau^2+h).
\end{equation}
\end{proof}
Detailed numerical studies that confirm the optimal convergence rates stated in Theorem \ref{main-theorem-2} and Theorem \ref{MainResult2} are presented in \cite{NonlinearCN,HW18}.

	\section{Implementation and Numerical Examples}
		
In this section we will discuss how the Crank-Nicolson FEM discretization can be efficiently implemented and practically used. Afterwards, we present two numerical experiments. The first one is to confirm the theoretically predicted convergence rates in Theorem \ref{main-theorem-2} and the second experiment demonstrates that our approach is fully competitive in low regularity regimes, where we compare it with a time-splitting spectral method.

\subsection{Efficient implementation}

	The Crank Nicolson method \eqref{FullyDiscrete}, albeit popular, suffers from the drawback that it requires solving a fully nonlinear system of equations in each time step. Furthermore, this system of equations is often solved through a Newton step, the implementation of which can become complicated and expensive for general nonlinearities. We present here a competitive fixed point solver which makes the method perform on par in terms of computational time with linearized time-stepping methods such as the RE-FEM proposed by C. Besse \cite{Besse} which was found to be best performing in \cite{HW18}.  
	
	To detail the proposed fixed-point iteration, let $U^n \in \mathbb{R}^N$ denote the vector of nodal values that belongs to the function $u^n_{\tau,h} \in S_h$. Introducing the following matrix notation:
	\begin{eqnarray*}
		&(M)_{ij} = \langle v_j , v_i \rangle &\quad  A_{ij} = \langle \nabla v_j , \nabla v_i \rangle \\
		&(M_V)_{ij} = \langle V v_j , v_i \rangle &\quad (M_\Gamma) _{ij}(U^{n+1},U^n) = \langle \tfrac{\Gamma(|U^{n+1}|^2) - \Gamma(|U^{n}|^2)}{|U^{n+1}|^2 - |U^{n}|^2} v_j , v_i \rangle
	\end{eqnarray*}
the equation \eqref{FullyDiscrete} in matrix form becomes :
	\begin{align*}
	\ci M\frac{U^{n+1}-U^n}{\tau} =\big( A+M_V+M_\Gamma(U^{n+1},U^n) \big)\frac{U^{n+1}+U^n}{2}.
	\end{align*}
	Let $L_1 = M+\ci\tau/2(A+M_V)$ and $L_2 = M-\ci \tau/2(A+M_V)$. Our fixed point iteration takes the form:
	\begin{equation}
		U^{n+1}_{i+1} = L_1^{-1}L_2U^n 
		- \ci\tau L_1^{-1}M_\Gamma (U^{n+1}_i,U^n)(U^{n+1}_i+U^n)/2.
	\end{equation} 
		Here we note that matrix $L_1$ does not change with time. Hence, the above iteration can be done efficiently by precomputing the LU-factorization of $L_1$. After it is precomputed, each time step only involves matrix-vector multiplications, but no longer the solving of a linear system of equations. In fact, the main cost in each time step account for the assembly of the updated mass matrix $M_\Gamma (U^{n+1}_i,U^n)$ with the densities from the previous time step and the previous iteration. Compared to this, all other costs are essentially negligible. We find that typically, but dependent on $\tau$, 4-8 iterations are required to reach a tolerance of machine epsilon.  To illustrate the efficiency we conclude with two numerical test problems. 

\subsection{Harmonic potential}
 First we consider a smooth potential to confirm the expected convergences rates for a type of nonlinearity that complements previous test cases \cite{NonlinearCN,HW18}. Here we seek $u(x,t)$ with
		\begin{gather}  \label{Saturated_Smooth}
		\begin{cases}
		\ci \partial_t u  &= -\Delta u + V u  +  \gamma(|u|^2)u   \qquad\hspace{12pt} \mbox{in } \mathcal{D} \times (0,T], \\
		u(\cdot,t)  &=  0     \qquad\hspace{111pt} \mbox{on } \partial \mathcal{D} \times (0,T],  \\
		u(\cdot,0)  &= \uzero  \qquad\hspace{105pt} \mbox{in }  \mathcal{D},
		\end{cases}
		\end{gather}
		where we consider the saturated nonlinearity $\gamma(r) := r/(1+r)$ (cf. \cite{Max76,MNF89,RTZ00}) . Furthermore, $\mathcal{D} = [-5,5]^2$ is the computational domain, the maximum time is selected as $T = 1$ and the trapping potential $V(x,y)  = (\nu_x x)^2+(\nu_y y)^2$.
		For the time-dependent problem we set the trapping frequencies to $\nu_x = 2$ and $\nu_y = 3$. The initial value $\uzero$ is the unique positive ground state with $\int_{\mathcal{D}} |\uzero|^2 =1$ to the problem with $\nu_x=\nu_y = 1$, i.e. it solves the eigenvalue problem
		\[\lambda_0 \uzero = -\Delta \uzero +V\uzero +\gamma(|\uzero|^2)\uzero, \]
		with ground state eigenvalue (chemical potential) $\lambda_0>0$.
	The $H^1$-errors are presented in Table \ref{H1-Errors}. The $\mathcal{O}(h)$-convergence is best seen in column $\tau=2^{-9}$, where initially the convergence is $\mathcal{O}(h^{1.5})$ but flattens out to $\mathcal{O}(h^{1.2})$ for the last data point. Since the reference solution is also computed with $h=0.0125$, it is expected that this last order of convergence is an overestimate. Using the values in row $h=0.0125$ we estimate the order of convergence with respect to $\tau$ to be $1.9$, hence, confirming the theoretically predicted rates from Theorem \ref{main-theorem-2}.
		\begin{table}[H]
		\begin{center}{$\|\nabla(u_{\tau,h}-u_{\text{ref}})\|_{L^2}$ }\\
		\begin{tabular}{l| l*{6}{c}}
             & $\tau = 2^{-5}$ & $\tau = 2^{-6}$ & $\tau = 2^{-7}$ & $\tau = 2^{-8}$ & $\tau = 2^{-9}$  & $\tau = 2^{-10}$ & $\tau = 2^{-11} $  \\
\hline
$h = 0.2$ & 0.829 & 0.400 & 0.485 & 0.526 & 0.538 & 0.542 & 0.543 \\
$h = 0.1$ & 1.048 & 0.396 & 0.144 & 0.179 & 0.191 & 0.194 & 0.194  \\
$h=0.05$  & 1.109 & 0.498 & 0.129 & 0.054 &  0.062 & 0.066 & 0.066 \\
$h = 0.025$  &  1.124 & 0.526 & 0.155 & 0.037 &  0.022 & 0.023 &  0.023  \\
$h=0.0125$  & 1.128  & 0.533 & 0.163 & 0.039 &  0.009 & 0.002 &\ 0.0005  \\
\end{tabular}
\end{center}
\caption{$H^1$-errors for harmonic potential test case \eqref{Saturated_Smooth}. The reference solution $u_{\text{ref}}$ is calculated with $h = 0.0125$ and $\tau=2^{-13}$. The energy is $E[u^n_{h,\tau}] = 3.86874$.} \label{H1-Errors}
\end{table}

\subsection{Discontinuous potential}
This example illustrates a moderate setting where, due to reduced regularity of the exact solution, finite element based methods are preferable over spectral methods. In the following we compare the Crank-Nicolson approach with a Strang splitting spectral method of order 2 (SP2) \cite{Spectral} which is known to show a very good performance in smooth settings.

In this test problem we seek $u(x,t)$ with
		\begin{gather}  \label{Saturated_Discont}
		\begin{cases}
		\ci \partial_t u  &= -\Delta u + V u  +  \gamma(|u|^2)u   \qquad\hspace{13pt} \mbox{in } \mathcal{D} \times (0,T], \\
		u(\cdot,0)  &= \uzero  \qquad\hspace{106pt} \mbox{in }  \mathcal{D},
		\end{cases}
		\end{gather}
		where we consider the saturated nonlinearity $\gamma(r) := 10r/(1+r)$. For a fair comparison with the SP2, we consider our problem with periodic boundary conditions which are easier to handle by the spectral method. The generalization of the Crank-Nicolson method to periodic boundary conditions is straightforward. Furthermore, $\mathcal{D} = [-5,5]^2$ is again the computational domain and the maximum time is selected as $T = 1$. The trapping potential $V(x,y)  = (\nu_x x)^2+(\nu_y y)^2+100(\mathds{1}_{|x|\geq 1}(x)+\mathds{1}_{|y|\geq 1}(y))$, with trapping frequencies $\nu_x = 1$ and $\nu_y = 3$, is discontinuous and causes a slight loss of regularity. We stress that this is a moderate test case, as illustrated in Fig. \ref{Density} most of the dynamics take place within the unit cube where the potential is smooth.  The initial value $\uzero$ is the unique positive ground state with $\int_{\mathcal{D}} |\uzero|^2 =1$ and $V_0(x,y)  = x^2+ y^2$, i.e. it solves the eigenvalue problem
		\[\lambda_0 \uzero = -\Delta \uzero +V_0\uzero +\gamma(|\uzero|^2)\uzero. \]

The errors and the computational times of the CN-FEM are presented in Table \ref{CN-FEM} and the errors and computational times of the SP2 in Table \ref{SP2}. The reference solution, $u_{\text{ref}}$, is computed using the CN-FEM with $h = 0.0125$ and $\tau = 2^{-13} $. The implantation was done in \verb|Julia|. It is important to keep in mind that the SP2 uses mostly inbuilt functions such as the fast Fourier transform from the \verb|C| subroutine library (\verb|FFTW|). These functions are heavily optimized and show an extremely good performance.
In spite of this we see,  comparing the errors of the CN-FEM for $h = 0.025$ and $\tau= 2^{-10}$ to those of the SP2 for $N_{DoF} = 3.2\cdot 10^{6}$ and $\tau = 2^{-13}$, that they are on par with respect to CPU time relative to accuracy, with a slight computational advantage for the CN-FEM. This advantage becomes clearer, the larger the region of reduced regularity (e.g. in the context of optical lattices or disorder potentials). This justifies the usage of CN-FEM in low regularity regimes. Furthermore, it is clearly seen that the space discretization dominates the error of the SP2, in order for the spectral method to catch up in terms of accuracy with the CN-FEM with $h=0.0125$, an estimated 10 to 40 million degrees of freedom would be needed and thus the memory cost would become an issue.

	\begin{figure}[H]
	\centering
\includegraphics[width=0.7\linewidth]{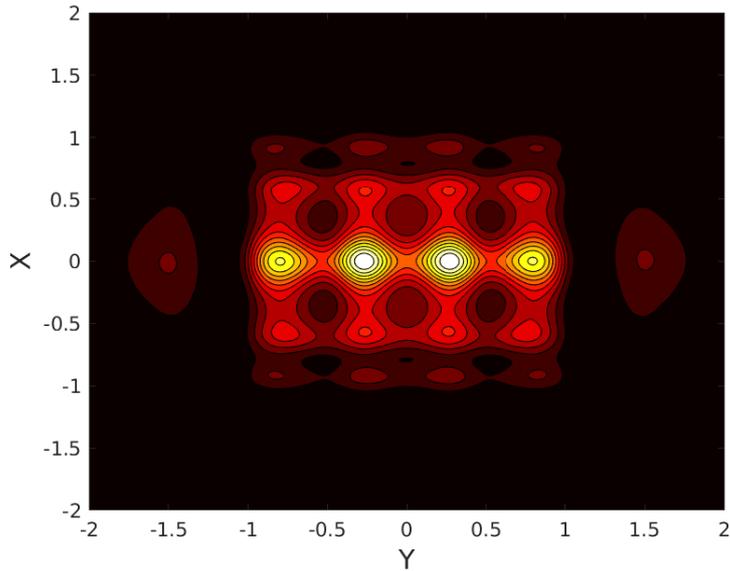}	
\caption{Plot of the density ($|u|^2$) of the reference solution to \eqref{Saturated_Discont} at $T=1$. }\label{Density}
\end{figure}

	\begin{table}[H]
		\centering
		\begin{tabular}{|c |c| c | c|  }
			\multicolumn{4}{c}{CN-FEM, $h=0.025$, $N_{DoF} = 400^2$}\\ \hline
			 & $\||u_{\tau,h}|^2-|u_{\text{ref}}|^2 \|_{L^1}$ & $\|\nabla (u_{h,\tau} -u_{\text{ref}}) \|_{L^2} $ & CPU [h]\\
			 \hline
			 $\tau = 2^{-8}$& 0.40 & 4.91 & 0.4 \\
			 $\tau = 2^{-9}$& 0.18 & 3.20 & 0.8 \\
			 $\tau = 2^{-10}$& 0.12 & 1.43 & 1.4 \\ \hline
			 
					\end{tabular}

		\begin{tabular}{|c |c| c | c|  }
			\multicolumn{4}{c}{CN-FEM, $h=0.0125$, $N_{DoF} = 800^2 $}\\ \hline
			& $\||u_{\tau,h}|^2-|u_{\text{ref}}|^2 \|_{L^1}$ & $\|\nabla (u_{h,\tau} -u_{\text{ref}}) \|_{L^2} $ & CPU [h]\\
			\hline
			$\tau = 2^{-8}$& 0.42 & 5.38 & 1.8 \\
			$\tau = 2^{-9}$& 0.22 & 3.61 & 3.3 \\
			$\tau = 2^{-10}$& 0.13 & 1.77 & 6.0 \\ 
			$\tau = 2^{-11}$& 0.04 & 0.80 & 11.5 \\ 
			\hline
		\end{tabular}
		\caption{Errors and computational times for the CN-FEM. For a relative comparison of the errors we recall $\|\ |u_{\tau,h}|^2 \ \|_{L^1}=1$ and note $E[u_{\tau,h}] = 25.539397$.} \label{CN-FEM}
	\end{table}

	\begin{table}[H]
	\centering
	\begin{tabular}{|c |c| c | c|  }
		\multicolumn{4}{c}{SP2, $N_{DoF}=800^2$}\\ \hline
		& $\||u_{\tau,h}|^2-|u_{\text{ref}}|^2 \|_{L^1}$ & $\|\nabla (u_{h,\tau} -u_{\text{ref}}) \|_{L^2} $ & CPU [h]\\
		\hline
		$\tau = 2^{-12}$& 0.28 & 4.29 & 0.24  \\
		$\tau = 2^{-13}$& 0.28 & 2.42  & 0.49  \\
		$\tau = 2^{-14}$& 0.28 & 2.35 & 0.96 \\ 
		\hline
	\end{tabular}

	\begin{tabular}{|c |c| c | c|  }
		\multicolumn{4}{c}{SP2, $N_{DoF}=1600^2$}\\ \hline
		& $\||u_{\tau,h}|^2-|u_{\text{ref}}|^2 \|_{L^1}$ & $\|\nabla (u_{h,\tau} -u_{\text{ref}}) \|_{L^2} $ & CPU [h]\\
		\hline
		$\tau = 2^{-12}$& 0.12 & 2.62 &  1.15\\
		$\tau = 2^{-13}$& 0.12 & 1.28 &  2.26 \\
		$\tau = 2^{-14}$& 0.12 & 1.18 &  5.2 \\ 
		\hline
	\end{tabular}
	\caption{Errors and computational times for the SP2. For a relative comparison of the errors we recall $\|\ |u_{\tau,h}|^2 \ \|_{L^1}=1$ and note $E[u_{\tau,h}] = 25.539397$.} \label{SP2}
\end{table}

\medskip
$\\$
{\bf Acknowledgements.}
We thank the anonymous referees for their helpful and insightful comments that improved the contents of this paper.

\appendix

\section{Explicit decomposition of the consistency error}\label{consistency-error-appendix}
Here we make the consistency error \eqref{Taylor} and its estimate \eqref{bounds-for-Tk} explicit and highlight where the regularity assumptions come in. For this we use the following standard integral remainder of Taylor expansion (for Sobolev functions $v\in H^{n+1}(a,b)$):
\begin{eqnarray*}
v(b)-T_n^{a,b}(v)  = \int_a^b\frac{(b-t)^n}{n!}v^{(n+1)}(t)dt ,
\end{eqnarray*}
where $a$ is the point of expansion and $T_n^{a,b}(v)$ the Taylor polynomial of degree $n$ of $v$ evaluated in $a$ and $b$. Recalling that $u^{n+1/2}=(u^{n+1} + u^n)/2$, the Taylor expansion implies in our case:

\begin{eqnarray*}
&&\| V(u^{n+1/2}-u(t_{n+1/2}))\|^2 = \frac{1}{2} \| V\big( \int_{t_{n+1/2}}^{t_{n+1}}(t_{n+1}-s)\partial_{tt}u(s)ds 
-\int_{t_{n+1/2}}^{t_{n}}(t_{n}-s)\partial_{tt}u(s)ds\big)\|^2\\ 
&&  \lesssim  (\int_{t_n}^{t_{n+1}}(t_{n+1}-s)\|\partial_{tt}u(s)\|ds\big)^2 \lesssim 
\big(\tau^{3/2} \|\partial_{tt}u(s)\|_{L^2( (t_n,t_{n+1}); L^2(\mathcal{D}))}\big)^2 \lesssim \tau^3 \int_{t_n}^{t_{n+1}}\|\partial_{tt}u\|^2 ds
\end{eqnarray*}
Thus \[ \sum_{k=1}^{n}\|u^{n+1/2}-u(t_{n+1/2})\|^2 \leq \tau^3 \int_0^T \|\partial_{tt}u(s)\|^2 ds \leq \tau^3 \|\partial_{tt}u\|^2_{L^2((0,T);L^2(\mathcal{D}) )}.  \]
Likewise we have
\begin{eqnarray*}
&&\|\ci \hspace{2pt}( D_\tau u^n - \partial_t u(t_{n+1/2}))\|^2 =\frac{1}{\tau^2} \| \int_{t_{n+1/2}}^{t_{n+1}}\frac{(t_{n+1}-s)^2}{2}\partial_{ttt}u(s)ds  - \int_{t_{n+1/2}}^{t_{n}}\frac{(t_{n}-s)^2}{2}\partial_{ttt}u(s)ds\|^2 \\
&& \lesssim \frac{1}{\tau^2} \tau^5 \int_{t_n}^{t_{n+1}} \|\partial_{ttt}u(s)\|^2ds 
\end{eqnarray*}
and
\begin{eqnarray*}
&& \|\Delta( u^{n+1/2}-u(t_{n+1/2}))\|^2 = \|\int_{t_{n+1/2}}^{t_{n+1}}\frac{(t_{n+1}-s)^2}{2} \Delta \partial_{tt}u(s) ds-\int_{t_{n+1/2}}^{t_{n}}\frac{(t_{n}-s)^2}{2} \Delta \partial_{tt}u(s) ds \|^2 \\
&&  \lesssim \tau^3 \int_{t_n}^{t_{n+1}}\| \Delta \partial_{tt} u(s)\|^2 ds.
\end{eqnarray*}
For the estimate of the term coming from the nonlinearity, we set for the sake of brevity $a:=|u^n|^2$, $b:= |u^{n+1}|^2$ and $c:=|u(t_{n+1/2})|^2$ to obtain
\begin{eqnarray*}
&&\|\frac{1}{|u^{n+1}|^2-|u^n|^2}\int_{|u^n|^2}^{|u^{n+1}|^2}\gamma(r)dr-\gamma(|u(t_{n+1/2})|^2 \|^2) =  \|\frac{1}{b-a}\int_a^b \gamma(r)dr -\gamma(c)\|^2 \\
&& = \|\frac{1}{b-a}\int_a^b \int^r_c \gamma'(s)dsdr\|^2 = \|\gamma'(c')\bigg(\frac{b+a}{2}-c\bigg)\|^2 \\ 
&& \lesssim \tau^3\| \partial_{tt}|u|^2 \|_{L^2((t_n,t_{n+1});L^2)}^2   \lesssim \tau^3 \| \partial_{tt} u \|_{L^2((t_n,t_{n+1});L^2)}^2.
\end{eqnarray*}
Where $c'$ lies between $a$ and $b$. Thus \[ 
\sum_{k=0}^n\|T^k\|^2 \lesssim \tau^3 (\|\partial_{ttt}u \|_{L^2((0,T);L^2 )}^2+\|\partial_{tt}u\|_{L^2((0,T);H^2)}^2+\|\partial_{t}u\|_{L^2((0,T);H^2)}^2+\|u\|_{L^2((0,T);H^2)}^2), \]
which proves \eqref{bounds-for-Tk}. Finally we use the above expression for $\gamma(\xi^n)-\gamma(|u(t_{n+1/2})|^2)$, to bound   $\|\nabla(\gamma(\xi^k)-\gamma(|u(t_{k+1/2})|^2))\|^2$,
\begin{eqnarray*}
\|\nabla(\gamma(\xi^n)-\gamma(|u(t_{n+1/2})|^2))\|^2&=&\|\gamma''(c')\nabla c'(\frac{b+a}{2}-c)+\gamma'(c')\nabla(\frac{b+a}{2}-c)\|^2\\
& \lesssim&  \| \frac{|u^{n+1}|^2+|u^n|^2}{2}-|u(t_{n+1/2})|^2\|_{H^1}^2.
\end{eqnarray*}

\end{document}